\documentclass[a4paper,12pt]{article}
\usepackage[hdivide={*,15.5cm,*},vdivide={*,24cm,*}]{geometry}
\usepackage[latin1]{inputenc}
\usepackage[english]{babel}
\usepackage[centertags]{amsmath}
\usepackage{amssymb,amsthm,latexsym}
\usepackage{footnote,color,cite,
            ifpdf,enumitem}
\usepackage[colorlinks=true,linkcolor=blue,
            citecolor=blue,urlcolor=blue]{hyperref}
\newcommand{\A}{\mathbb{A}}
\newcommand{\R}{\mathbb{R}}
\newcommand{\C}{\mathbb{C}}
\newcommand{\Z}{\mathbb{Z}}
\newcommand{\Q}{\mathbb{Q}}
\newcommand{\N}{\mathbb{N}}
\newcommand{\V}{\mathcal{V}}
\newcommand{\rk}{\operatorname{rk}}
\newcommand{\xon}{X_1,\ldots,X_n}

\newcommand{\Fon}{F_1,\ldots,F_n}
\newcommand{\iFon}{(F_1,\ldots,F_n)}

\newcommand{\Goq}{G_1,\ldots,G_q}
\newcommand{\bfs}[1]{\boldsymbol{#1}}
\newcommand{\klk}{,\ldots,}
\newtheorem{proposition}{Proposition}
\newtheorem{theorem}[proposition]{Theorem}
\newtheorem{corollary}[proposition]{Corollary}
\newtheorem{lemma}[proposition]{Lemma}
\newtheorem{remark}[proposition]{Remark}
\newtheorem{claim}[proposition]{Claim}
\theoremstyle{remark}
\newtheorem{example}[proposition]{Example}

\ifpdf
\fi
\begin{document}
\selectlanguage{english}
\title{Degeneracy loci and polynomial equation solving$\,^{1}$}
\author{{\bf Bernd Bank$\,^{2}$, Marc Giusti$\,^{3}$, Joos
    Heintz$\,^{4}$,}\\ {\bf Gr\'egoire Lecerf$\,^{5}$, Guillermo
    Matera$\,^{6}$, Pablo Solern\'o$\,^{7}$ }}
%
%
\maketitle
\footnotetext{Communicated by Teresa Krick and James Renegar.}
\addtocounter{footnote}{1}\footnotetext{Research partially supported
  by the following Argentinian, French and Spanish grants: \\
  CONICET Res 4541-12, PIP 11220090100421 CONICET,
  UBACYT 20020100100945 and
  20020110100063, PICT--2010--0525, Digiteo DIM 2009--36HD ``MaGiX''
  grant of the R{\'e}gion Ile-de-France, ANR-2010-BLAN-0109-04
  ``LEDA'', MTM2010-16051.}
\addtocounter{footnote}{1}\footnotetext{Humboldt-Universit\"at zu
  Berlin, Institut f\"ur Mathematik, 10099 Berlin, Germany.
  \texttt{bank@mathematik.hu-berlin.de}}
\addtocounter{footnote}{1}\footnotetext{Laboratoire d'informatique
  LIX, UMR 7161 CNRS, campus de l'{\'E}cole polytechnique, 91128
  Palaiseau
  Cedex, France.
  \texttt{Marc.Giusti@Polytechnique.fr}}
\addtocounter{footnote}{1}\footnotetext{Departamento de Computaci\'on,
  Universidad de Buenos Aires and CONICET, Ciudad Univ., Pab.I, 1428
  Buenos Aires, Argentina, and Departamento de Matem\'aticas,
  Estad\'{\i}stica y Computaci\'on, Facultad de Ciencias, Universidad
  de Cantabria, 39071
  Santander, Spain.
  \texttt{joos@dc.uba.ar}}
\addtocounter{footnote}{1}\footnotetext{Laboratoire d'informatique
  LIX, UMR 7161 CNRS, campus de l'{\'E}cole polytechnique, 91128
  Palaiseau
  Cedex, France.
  \texttt{gregoire.lecerf@math.cnrs.fr}}
\addtocounter{footnote}{1}\footnotetext{Instituto del Desarrollo
  Humano, Universidad Nacional de General Sarmiento and CONICET,
  J. M. Gutierrez 1150 (B1613GSX) Los Polvorines, Buenos
  Aires, Argentina.
  \texttt{gmatera@ungs.edu.ar}}
\addtocounter{footnote}{1}\footnotetext{Instituto Matem\'atico Luis
  Santal\'o CONICET and Departamento de Matem\'aticas, Facultad de
  Ciencias Exactas y Naturales, UBA, 1428 Buenos Aires,
  Argentina.
  \texttt{psolerno@dm.uba.ar}}
\begin{center}
  Dedicated to Mike Shub on the occasion of his $70\,^{th}$
  birthday.
\end{center}
\bigskip
\begin{abstract}
  Let $V$ be a smooth equidimensional quasi-affine variety of
  dimension $r$ over $\C$ and let $F$ be a $(p\times s)$-matrix of
  coordinate functions of $\C[V]$, where $s\ge p+r$. The pair $(V,F)$
  determines a vector bundle $E$ of rank $s-p$ over $W:=\{x\in V \mid
  \rk F(x)=p\}$. We associate with $(V,F)$ a descending chain of
  degeneracy loci of $E$ (the generic polar varieties of $V$ represent
  a typical example of this situation).
  \par
  The maximal degree of these degeneracy loci constitutes the
  essential ingredient for the uniform, bounded error probabilistic
  pseudo-polynomial time algorithm which we are going to design and
  which solves a series of computational elimination problems that can
  be formulated in this framework. We describe applications to
  polynomial equation solving over the reals and to the computation of
  a generic fiber of a dominant endomorphism of an affine space.
\end{abstract}
\section{Introduction}
\label{s:0}
Let $V$ be a smooth and equidimensional quasi-affine variety over $\C$
of dimension $r$ and let $F$ be a $(p\times s)$-matrix of coordinate
functions of $\C[V]$, where $s\ge p+r $. Then $F$ determines a vector
bundle $E$ of rank $s-p$ over $W:=\{x\in V \mid \rk F(x)=p\}$. With
$E$ and a given \emph{generic} matrix $a\in \C^{(s-p)\times s}$ we
associate a descending chain of degeneracy loci of $E$. The generic
polar varieties constitute a typical example of this situation.
\par
We prove that these degeneracy loci are empty or equidimensional,
normal and Cohen-Macaulay. Moreover, if $b$ is another generic matrix,
the degeneracy loci associated with $a$ and $b$ are rationally
equivalent and their equivalence classes can be expressed in terms of
the Chern classes of $E$. Not the rational equivalence classes, but
the degeneracy loci themselves constitute a useful tool to solve
efficiently certain computational elimination tasks associated with
suitable quasi-affine varieties and matrices~$F$. Such elimination
tasks are for example real root finding in reduced complete
intersection varieties with a smooth and compact real trace or the
problem to describe efficiently a generic fiber of a given bi-rational
endomorphism of an affine space.

In a somewhat different context of effective elimination theory,
rational equivalence classes of degeneracy loci were considered in
\cite{BuLo07}.
\subsection{Contributions}
The main contribution of this paper is a new algorithm which solves
the above mentioned and other elimination tasks in uniform, bounded
error probabilistic \emph{pseudo-polynomial time}. In this sense it
belongs to the pattern of elimination procedures which became
introduced in symbolic semi-numeric computation by the already
classical Kronecker algorithm~\cite{GiHeHaeMorMonPa97, GiHeMorMoPa98,
  HeMaWa01, GiLeSa01, CaMat06, DurLe08}. Here we refer to procedures
whose inputs are measured in the usual way by syntactical, extrinsic
parameters and, besides of them, by a semantical, intrinsic parameter
which depends on the geometrical meaning of the input and may become
exponential in terms of the syntactical parameters. A procedure is
called \emph{pseudo-polynomial} if its time complexity is polynomial
in both, the syntactical and semantical parameters. In this sense, the
semantical parameter that controls the complexity of our main
algorithm is the maximal degree of the degeneracy loci which we
associate with the given elimination task.
\par
The particular feature of this algorithm is that the input polynomials
of the elimination task under consideration may be given by an
essentially division-free arithmetic circuit (which means that only
divisions by scalars are allowed) of size $L$. The complexity of the
algorithm becomes then of order $L(snd)^{O(1)}\delta^2$, where $n$ is
the number of indeterminates of the input polynomials, $d$ their
maximal degree, $s$ the number of columns of the given matrix and
$\delta$ is essentially the maximal degree of the degeneracy loci
involved. In worst case this complexity is of order
$(s(nd)^n)^{O(1)}$. General degeneracy loci constitute an important
instance where we are able to achieve, as a generalization
of~\cite{GiLeSa01}, a complexity bound of order square~$\delta$. At
present no other elimination procedure reaches such a sharp bound. In
particular we do not rely on equidimensional decomposition whose best
known complexity is of order cube~$\delta$
(see~\cite[Theorem~8]{SafeyTrebuchet2006} for an application to polar
varieties). For comparisons with the complexity of Gr\"obner basis
algorithms we refer to~\cite{Bardet2004}. Without going into the
technical details we indicate also how our algorithm may be realized
in the non-uniform deterministic complexity model by algebraic
computation trees. We implemented our main algorithm within the
\texttt{C++} library \texttt{geomsolvex} of
\textsc{Mathemagix}~\cite{Mathemagix}.
\par
In Section~\ref{s:1} we present some of the basic mathematical facts
concerning the geometry of our degeneracy loci which will be used in
Section~\ref{s:2} to develop our main algorithm. The proofs, which
are all self-contained except one, require only some knowledge of
classical algebraic geometry and commutative algebra which can be
found, e.g., in~\cite{Mumford88,
  Shafarevich94, Matsumura86}, elementary properties of vector bundles
over algebraic varieties~\cite{ShafarevichII94}, the Thom-Porteous
formula~\cite[Chapter~14]{Fulton84}, and the notion of linear
equivalence of cycles~\cite[Chapter~1]{Fulton84}. The main algorithm
requires some familiarity with the classical version of the
Kronecker algorithm~\cite{GiLeSa01,DurLe08} and with algebraic
complexity~\cite{BuClSh97}.
\subsection{Notions and notations}
\label{ss:0.1}
We shall use freely standard notions, notations and results of
classic algebraic geometry, commutative algebra and algebraic
complexity theory which can be found, e.g., in the
books~\cite{Mumford88, Shafarevich94, Matsumura86, BuClSh97}.
\par
Let $\Q$ and $\C$ be the fields of rational and complex numbers, let
$\xon$ be indeterminates over $\C$ and let be given polynomials
$\Goq$, and $H$ in $\C[\xon]$. By $\A^n$ we denote the
$n$-dimensional affine space over $\C$. We shall use the following
notations:
\[
\{G_1=0\klk G_q=0\}:=\{x\in \A^n \mid G_1(x)=0\klk G_q(x)=0\}
\]
and
\[
\{G_1=0\klk G_q=0\}_H:=\{x\in \A^n \mid G_1(x)=0\klk G_q(x)=0,
H(x)\neq 0\}.
\]
Suppose $1\le q \le n$ and that $\Goq$ form a regular sequence in
the localized ring $\C[\xon]_H$. We call it \emph{reduced} outside of $\{H=0\}$ if for
any index $1\le k\le q$ the ideal $(G_1\klk G_k)_H$ is radical in
$\C[\xon]_H$. Let $V$ be the quasi-affine subvariety of the ambient
space $\A^n$ defined by $G_1=0\klk G_q=0$ and $H\neq 0$,
\emph{i.e.},
\[
V:=\{G_1=0\klk G_q=0\}_H.
\]
By $\C[V]$ we denote the coordinate ring of $V$ whose elements we call
the coordinate functions of $V$. We adopt the same notations of $V$ as
we did for $V:=\A^n$.
\par
Suppose for the moment that $V$ is a closed subvariety of $\A^n$,
\emph{i.e.}, $V$ is of the form $V=\{G_1=0\klk G_q=0\}$. For $V$
irreducible we define its degree $\deg V$ as the maximal number of
points we can obtain by cutting $V$ with finitely many affine
hyperplanes of $\C^n$ such that the intersection is finite. Observe
that this maximum is reached when we intersect $V$ with dimension of
$V$ many \emph{generic} affine hyperplanes of $\C^n$. In case that
$V$ is not irreducible let $V=C_1\cup \cdots \cup C_s$ be the
decomposition of $V$ into irreducible components.  We define the
degree of $V$ as $\deg V:=\sum _{1\le j \le s}\deg C_j$. With this
definition we can state the so-called \emph{B\'ezout Inequality}: if
$V$ and $W$ are closed subvarieties of $\C^n$, then we have
\[\deg (V\cap W)\le \deg V\cdot \deg W.\]
If $V$ is a hypersurface of $\C^n$ then its degree equals the degree
of its minimal equation. The degree of a point of $\C^n$ is just
one. For more details we refer to~\cite{Heintz83, Fulton84, Vogel84}.
\section{Degeneracy loci}
\label{s:1}
We present the mathematical tools we need for the design of our main
algorithm in Section~\ref{s:2}. Proposition~\ref{p:1} and Theorem
\ref{t:2} below are not new. They can be extracted from existing
results of modern algebraic geometry. Since we use the ingredients of
our argumentation for these statements otherwise, we give new
elementary proofs of them. This makes our exposition self-contained.
%

Let $V$ be a quasi-affine variety and suppose that $V$ is smooth and
equidimensional of dimension $r$. The following constructions,
statements and proofs generalize basic arguments of~\cite{BaGiHePa04,
  BaGiHePa05, BaGiHeSaSc10}. Let $p$ and $s$ be natural numbers with
$s\ge p+r$. We suppose that there is given a $(p\times s)$-matrix of
coordinate functions of $V$, namely
\[
F:=\begin{bmatrix}
f_{1,1}&\cdots & f_{1,s}\\
\vdots&      & \vdots\\
f_{p,1}&\cdots & f_{p,s}
\end{bmatrix}\in \C[V]^{p\times s}.
\]
For $x\in V$ we denote by $\rk F(x)$ the rank of the complex $(p\times
s)$-matrix $F(x)$. Let $W:=\{x\in V \mid \rk F(x)=p\}$ and observe
that $W$ is an open, not necessarily affine, subvariety of $V$ which
is covered by canonical affine charts given by the $p$-minors of~$F$.

Let $E:= \{(x,y)\in W\times \A^s \mid F(x)\cdot y^T=0\}$ and $\pi:E
\to W$ be the first projection (here~$y^T$ denotes the transposed
vector of $y$). One sees easily that $\pi$ is a vector bundle of
rank $s-p$. We call $\pi$ (or $E$) the vector bundle associated with
the pair $(V,F)$. Let us fix for the moment a complex $((s-p)\times
s)$-matrix
\[
a:=\begin{bmatrix}
a_{1,1}&\cdots & a_{1,s}\\
\vdots&      & \vdots\\
a_{s-p,1}&\cdots & a_{s-p,s}
\end{bmatrix}\in \C^{(s-p)\times s}
\]
with $\rk a=s-p$. For $1\le i \le r+1$, let
\[
a_i:=\begin{bmatrix}
a_{1,1}&\cdots & a_{1,s}\\
\vdots&      & \vdots\\
a_{s-p-i+1,1}&\cdots & a_{s-p-i+1,s}
\end{bmatrix}\in \C^{(s-p-i+1)\times s}.
\]
We have $\rk a_i=s-p-i+1$. Let
\[
 T(a_i):=
\begin{bmatrix}
F\\a_i
\end{bmatrix} \in \C[V]^{(s-i+1)\times s}\quad\text{and}\quad
W(a_i):= \{x\in W \mid \rk T(a_i)(x)< s-i+1\}.
\]
Applying~\cite[Theorem~3]{EaNo62}
or~\cite[Theorem~13.10]{Matsumura86} to each canonical affine chart
of $W$ we conclude that any irreducible component of $W(a_i)$ has
codimension at most $i$ in~$W$. For $1\le i \le r$, the locally
closed algebraic varieties $W(a_i)$ form a descending chain
\[W \supseteq W(a_1) \supseteq \cdots \supseteq W(a_r). \]
We call the algebraic variety $W(a_i)$ the \emph{$i$-th degeneracy
  locus} of the pair $(V,F)$ associated with $a$.
\par
The vector bundle $E$ is a subbundle of $W\times \A^s$. Fix $1\le i
\le r$. Then the matrix $a_i$ defines a bundle map $W\times \A^s \to
W\times \A^{s-p-i+1}$ which associates with each $(x,y)\in W\times
\A^s$ the point $(x,a_i\cdot y^T)\in W\times \A^{s-p-i+1}$. By
restriction we obtain a bundle map $\varphi_i:E\to W\times
\A^{s-p-i+1}$ whose critical locus we are going to identify with
$W(a_i)$. First we observe that $(x,y)\in E$ is a critical point of
$\varphi_i$ if, and only if, any point of the fiber $E_x$ of $E$ at
$x$ is critical for $\varphi_i$. Thus the property of being a critical
point of $\varphi_i$ depends only on the fiber. We say that $x\in W$
is \emph{critical} for $\varphi_i$ if this map is critical on
$E_x$. One verifies easily by direct computation that the degeneracy
locus $W(a_i)$ is the set of critical points of $W$ for
$\varphi_i$. In this sense $W(a_i)$ is a degeneracy locus of
$\varphi_i$~\cite[Chapter~14]{Fulton84}.
\begin{example}
  We are going to visualize our setup by a simple example. Consider
  the polynomial $G:= X_1^2 + X_2^2 + X_3^2 - 1\in\C[X_1,X_2,X_3]$.
  Then $V:= \{ G = 0 \}$ is an irreducible subvariety of $\A^3$ which
  is smooth of dimension $r:=2$. Let $F$ be the gradient of $G$
  restricted to $V$, and let $p:= 1$ and $s:= 3$. Thus $s=p+r$. For
  $\pi_1, \pi_2, \pi_3$ being the coordinate functions of $\C[V]$
  induced by $X_1,X_2,X_3$ and for
  \[
   \begin{bmatrix}
    a_{1,1}& a_{1,2}& a_{1,3}\\
    a_{2,1}& a_{2,2}& a_{2,3}
  \end{bmatrix}\in\C^{2\times 3}
  \]
  generic, we have
  \[ F =
  \begin{bmatrix}
    2 \pi_1 & 2 \pi_2 & 2 \pi_3
  \end{bmatrix}, \qquad
  T(a_1) =
  \begin{bmatrix}
    2 \pi_1 & 2 \pi_2 & 2 \pi_3 \\
    a_{1,1}& a_{1,2}& a_{1,3}\\
    a_{2,1}& a_{2,2}& a_{2,3}
  \end{bmatrix}, \qquad
  T(a_2) =
  \begin{bmatrix}
    2 \pi_1 & 2 \pi_2 & 2 \pi_3 \\
    a_{1,1}& a_{1,2}& a_{1,3}
  \end{bmatrix}.
  \]
  One verifies easily that $W= V$, $W(a_1)= \{ x \in V \mid \det
  (T(a_1))= 0 \}$ and $W(a_2)= \{ x \in V \mid x=(x_1,x_2,x_3),
  a_{1,2} x_1 - a_{1,1} x_2 = 0, a_{1,3} x_1 - a_{1,1} x_3 = 0,
  a_{1,3} x_2 - a_{1,2} x_3 = 0\}$ holds. Since the matrix
  $[a_{i,j}]_{1\le i\le 2,1\le j\le 3}$ is generic by assumption, we conclude
  that $W(a_1)$ is equidimensional of dimension one and that
  $W(a_2)$ is the classical polar variety of the sphere $V$, which
  can be parameterized in the following way:
  \[
  W(a_2)= \left\{ \left(\frac{a_{1,1}^2}{a_{1,3}^2} +
    \frac{a_{1,2}^2}{a_{1,3}^2} + 1 \right) X_3^2 - 1 = 0,
  X_1- \frac{a_{1,1}}{a_{1,3}} X_3=0, X_2- \frac{a_{1,2}}{a_{1,3}} X_3=
  0 \right\}.
  \]
\end{example}
\subsection{The dimension of a degeneracy locus}
\label{s:1.2}
We are now going to show that, for a \emph{generic} matrix $a$, the
degeneracy locus $W(a_i)$ is either empty or of expected pure
codimension $i$ in $W$ (see Proposition~\ref{p:1} below). Our
considerations will only be local. Therefore it suffices to consider
the items we are going to introduce now. Let
\[
\Delta:=\det \begin{bmatrix}
f_{1,1}&\cdots & f_{1,p}\\
\vdots &   &\vdots\\
f_{p,1}&\cdots & f_{p,p}\end{bmatrix}.
\]
For $1\le i \le r$, let
\[
m_i:=\det \begin{bmatrix}
f_{1,1}&\cdots & f_{1,s-i}\\
\vdots &    &\vdots\\
f_{p,1}&\cdots & f_{p,s-i}\\
a_{1,1}&\cdots & a_{1,s-i}\\
\vdots &     &\vdots\\
a_{s-p-i,1}&\cdots & a_{s-p-i,s-i}\\
\end{bmatrix}.
\]
Thus $m_i$ is the upper-left corner $(s-i)$-minor of the
$((s-i+1)\times s)$-matrix $T(a_i)$. Further, let
\[M_{s-i+1}\klk M_{s}\] be the $(s-i+1)$-minors of the matrix $T(a_i)$
given by the columns numbered $1\klk s-i$ to which we add, one by one,
the columns numbered $s-i+1 \klk s$. Observe that
\[
{W(a_i)}_{\Delta}:=\{x\in W(a_i) \mid \Delta(x)\neq 0\}
\]
is an affine chart of the degeneracy locus $W(a_i)$. The
Exchange Lemma in~\cite{BaGiHeMb01} implies
\[
{W(a_i)}_{\Delta \cdot m_i}=\{ M_{s-i+1}=0 \klk
M_{s}=0\}_{\Delta\cdot m_i}.
\]
Let $Z_{s-i+1}\klk Z_{s}$ be new indeterminates and
$\widetilde{M}_{s-i+1}\klk \widetilde{M}_{s}$ be the $(s-i+1)$-minors
of the matrix
\[
\begin{bmatrix}
f_{1,1}& \cdots & f_{1,s-i} & f_{1,s-i+1}& \cdots & f_{1,s}\\
\vdots & & \vdots & \vdots & &\vdots\\
f_{p,1}& \cdots & f_{p,s-i} & f_{p,s-i+1} &\cdots & f_{p,s}\\
a_{1,1} & \cdots & a_{1,s-i} & a_{1, s-i+1}& \cdots & a_{1,s}\\
\vdots & & \vdots & \vdots & &\vdots\\
a_{s-p-i,1}& \cdots & a_{s-p-i, s-i} &  a_{s-p-i, s-i+1} & \cdots & a_{s-p-i,s}\\
a_{s-p-i+1,1} & \cdots & a_{s-p-i+1,s-i}& Z_{s-i+1} &\cdots & Z_s
\end{bmatrix}
\]
given by the columns numbered $1\klk s-i$ to which we add, one by one,
the columns numbered $s-i+1 \klk s$. We consider now the morphism
$\Phi_i:V_{m_i}\times \A^i \longrightarrow \A^i$ of smooth algebraic
varieties defined for $x\in V_{m_i}$ and $z\in \A^i$ by $(x,z)\mapsto
\Phi_i(x,z):=(\widetilde{M}_{s-i+1}(x,z)\klk \widetilde{M}_{s}(x,z)).$
\begin{lemma}\label{l:1}
  The origin $(0\klk 0)$ of $\A^i$ is a regular value of $\Phi_i$.
\end{lemma}
\begin{proof}
  Without loss of generality we may assume that $\Phi_i^{-1}(0\klk 0)$
  is nonempty.  Let $x\in V_{m_i}$ and $z\in \A^i$ with
  $\Phi_i(x,z)=(0\klk 0)$ be arbitrarily chosen.  Observe that the
  Jacobian of $\Phi_i$ at $(x,z)$ is a matrix with $i$ rows of the
  following form:
  \[
  \begin{bmatrix}
    *& \cdots & * & m_i(x) &0&\cdots & 0\\
    *& \cdots & * & 0 & m_i(x)&\ddots &\vdots\\
    \vdots&  & \vdots& \vdots& \ddots& \ddots & 0 \\
    *& \cdots & * & 0 & \cdots& 0& m_i(x)
  \end{bmatrix}.
  \]
  Since $x$ belongs to $V_{m_i}$ we conclude that $(x,z)$ is a regular
  point of $\Phi_i$. The arbitrary choice of $(x,z)$ in
  $\Phi_i^{-1}(0\klk0)$ implies now Lemma~\ref{l:1}.
\end{proof}
{From} the Weak Transversality Theorem of Thom-Sard (see, e.g.,
\cite[Theorem~III.7.4]{Demazure89}) we deduce now that there exists
a nonempty Zariski open set $\Omega$ of $\A^i$ such that for any
point $z\in \Omega$ the equations $\widetilde{M}_{s-i+1}(x, z)=0\klk
\widetilde{M}_{s}(x, z)=0$ intersect transversally at any common
zero belonging to $V_{m_i}$. From now on we shall choose the complex
$((s-p)\times s)$-matrix $a$ generically proceeding step by step
from row numbered one until row numbered $s-p$. With this choice in
mind we may suppose without loss of generality that the equations
$M_{s-i+1}=0\klk M_{s}=0$ intersect transversally at any of their
common zeros belonging to $V_{m_i}$. In particular,
${W(a_i)}_{\Delta\cdot m_i}=\{M_{s-i+1}=0\klk M_{s}=0\}_{\Delta\cdot
  m_i}$ is either empty or of pure codimension $i$ in $W_{\Delta\cdot
  m_i}$.
\begin{proposition}[\!\!{\cite[Transversality Lemma~1.3~(i)]{Piene78}}]
\label{p:1} For $1\le i \le r$ and a generic matrix $a\in \mathbb{C}^{(s-p)\times s}$, the
  $i$-th degeneracy locus $W(a_i)$ is empty or of pure codimension $i$
  in $W$.
\end{proposition}
\begin{proof}
  Let $C$ be an irreducible component of $W(a_i)$ not contained in
  $W(a_{i+1})$. Without loss of generality we may assume that
  $\Delta\cdot m_i$ does not vanish identically on $C$. Therefore
  $C_{\Delta\cdot m_i}$ is an irreducible component of
  $W(a_i)_{\Delta\cdot m_i}$. Hence, $C_{\Delta\cdot m_i}$ is of
  codimension $i$ in $W_{\Delta\cdot m_i}$. This implies that the
  codimension of $C$ in $W$ is also $i$.
  \par
  Let us consider the case $i=r$. By induction on $1\le j \le s-p-r$
  we conclude in the same way as in the proof of Lemma~\ref{l:1} and
  the observations following it that for any point $x$ of $W_{\Delta}$
  there exists a $(p+j)$-minor corresponding to $p+j$ columns,
  including those numbered $1\klk p$, of the matrix
  \[
  \begin{bmatrix}
    &F&\\
    a_{1,1}&\cdots & a_{1,s}\\
    \vdots&      & \vdots\\
    a_{j,1}&\cdots & a_{j,s}\\
  \end{bmatrix}
  \]
  which does not vanish at $x$. This implies that
  $W(a_{r+1})_{\Delta}$ is empty. Thus $W(a_r)_{\Delta}$, and hence
  $W(a_r)$, is empty or of pure codimension $r$ in $W$. This shows
  Proposition~\ref{p:1} in case $i=r$.
  \par
  Suppose now that Proposition~\ref{p:1} is wrong and let $1\le i < r$
  be maximal such that there exists an irreducible component $C$ of
  $W(a_i)$ with codimension different from $i$ in $W$. Then $C$ must
  be contained in $W(a_{i+1})$. There exists an irreducible component
  $D$ of $W(a_{i+1})$ with $D\supseteq C$. From the maximal choice of
  $i$ we deduce that the codimension of $D$ in $W$ is $i+1$. This
  implies that the codimension of $C$ in $W$ is at least $i+1$.  On
  the other hand, we have seen that the codimension of $C$ in $W$ is
  at most $i$.  This contradiction implies Proposition~\ref{p:1}.
\end{proof}
By the way we have proved that the variety $W(a_i)\setminus
W(a_{i+1})$ is empty or equidimensional and smooth and that it can
be defined locally by reduced complete intersections.
\begin{corollary}[\!\!{\cite[Theorem~14.4~(c)]{Fulton84}}]
  \label{c:1} For $1\le i \le r$ and a generic matrix $a\in \mathbb{C}^{(s-p)\times s}$, the
  degeneracy locus $W(a_i)$ is empty or equidimensional and
  Cohen-Macaulay.
\end{corollary}
\begin{proof}
  The statement is local. So we may, without loss of generality,
  restrict our attention to the affine variety $W(a_i)_{\Delta}$ and
  may suppose $W(a_i)_{\Delta}\neq \emptyset$.  Observe that the
  affine variety $W_{\Delta}$ is equidimensional and smooth and
  therefore Cohen-Macaulay. Furthermore, $W(a_i)_{\Delta}$ is a
  determinantal subvariety of $W_{\Delta}$ given by maximal minors
  which is by Proposition~\ref{p:1} of pure codimension $i$ in
  $W_{\Delta}$.
  \par
  Applying now~\cite[Theorem~2.7 and Proposition~16.19]{BrVe88} to
  this situation we conclude that $W(a_i)_{\Delta}$ is Cohen-Macaulay
  (see also~\cite{EaNo62, EaHo74} and~\cite[Section~18.5]{Eisenbud95}
  for the general context of determinantal varieties). This implies
  Corollary~\ref{c:1}.
\end{proof}
Taking into account Corollary \ref{c:1} we conclude that the
$(s-i+1)$-minors of $T(a_i)$ induce in the local ring of $W$ at any
point of $W(a_i)$ a radical ideal. Therefore $W(a_i)$ is as scheme
reduced.
\subsection{Normality and rational equivalence of degeneracy loci}
\begin{theorem}\label{t:2}
  For $1\le i \le r$ and a generic matrix $a\in \mathbb{C}^{(s-p)\times s}$, the degeneracy locus $W(a_i)$ is empty or
  equidimensional, Cohen-Macaulay and normal.
\end{theorem}
\begin{proof}
  Again, the statement of Theorem~\ref{t:2} being local, we may
  restrict our attention to the affine variety $W(a_i)_{\Delta}$ which
  we suppose to be nonempty. By Corollary~\ref{c:1}, the variety
  $W(a_i)_{\Delta}$ is equidimensional and Cohen-Macaulay and by
  Serre's normality criterion (see, e.g.,
  \cite[Theorem~23.8]{Matsumura86}) it suffices therefore to prove the
  following statement.
  \begin{claim}
    The singular points of $W(a_i)_{\Delta}$ form a subvariety of
    codimension at least two.
  \end{claim}
  \begin{proof}[Proof of the claim]
    We follow the general lines of the argumentation
    in~\cite[Section~3]{BaGiHeSaSc10}. In case $i=r$,
    Proposition~\ref{p:1} implies the claim. Let us therefore suppose
    that there exists an index $1\le i <r$ such that the claim is
    wrong. Let
    \[
    \upsilon:=\det \begin{bmatrix}
      f_{1,1}&\cdots & f_{1,s-i-1}\\
      \vdots &     &\vdots\\
      f_{p,1}&\cdots & f_{p,s-i-1}\\
      a_{1,1}&\cdots & a_{1,s-i-1}\\
      \vdots &     &\vdots\\
      a_{s-p-i-1,1}&\cdots & a_{s-p-i-1,s-i-1}\\
    \end{bmatrix}
    \]
    be the $(s-i-1)$-minor of the $((s-i-1)\times s)$-matrix
    $T(a_{i+2})=\begin{bmatrix}F\\a_{i+2}\end{bmatrix}$ which is given
    by the columns numbered $1\klk s-i-1$.
    \par
    For $s-p-i\le k \le s-p-i+1$ and $s-i\le l \le s$ let
    \[
    m_{k,l}:=\det \begin{bmatrix}
      f_{1,1}&\cdots & f_{1,s-i-1}&f_{1,l}\\
      \vdots &     &\vdots&\vdots\\
      f_{p,1}&\cdots & f_{p,s-i-1}&f_{p,l}\\
      a_{1,1}&\cdots & a_{1,s-i-1}&a_{1,l}\\
      \vdots &     &\vdots&\vdots\\
      a_{s-p-i-1,1}&\cdots & a_{s-p-i-1,s-i-1}&a_{s-p-i-1,l}\\
      a_{k,1}&\cdots & a_{k,s-i-1}&a_{k,l}\\
    \end{bmatrix}.
    \]
    We consider now an arbitrary point $x$ of $W(a_i)_{\Delta\cdot
      \upsilon}$. If there exists a pair $(k,l)$ of indices with
    $s-p-i\le k \le s-p-i+1$ and $s-i\le l \le s$ and $m_{k,l}(x)\neq 0$,
    then, by the generic choice of the complex $(r\times s)$-matrix
    $a$, the variety $W(a_i)_{\Delta\cdot \upsilon}$ must be smooth
    at~$x$ (compare to Lemma~\ref{l:1} and the comments following it).
    \par
    Therefore the singular locus of $W(a_i)_{\Delta\cdot \upsilon}$ is
    contained in
    \[
    \mathcal{Z}:=W_{\Delta\cdot \upsilon}\cap \{ m_{k,l}=0 \mid s-p-i\le k
    \le s-p-i+1,\; s-i\le l \le s\}.
    \]
    Again, the generic choice of $a$ implies that $\mathcal{Z}$ is
    empty or has pure codimension $2(i+1)$ in $W_{\Delta\cdot
      \upsilon}$.  Hence, the singular locus of $W(a_i)_{\Delta\cdot
      \upsilon}$ has at least codimension $2(i+1)$ in $W_{\Delta\cdot
      \upsilon}$ and therefore at least codimension two in
    $W(a_i)_{\Delta\cdot \upsilon}$. This argumentation proves that
    the singular points of $W(a_i)_{\Delta}\setminus
    W(a_{i+2})_{\Delta}$ are contained in a subvariety of
    $W(a_i)_{\Delta}$ of codimension at least two. Since
    $W(a_{i+2})_{\Delta}$ is empty or has by Proposition~\ref{p:1}
    codimension two in $W(a_i)_{\Delta}$ the claim follows.
  \end{proof}
  \noindent
  This ends our proof of Theorem~\ref{t:2}. For more details we refer
  to \cite{BaGiHeSaSc10}.
\end{proof}
\begin{corollary}\label{c:2}
  For $1\le i \le r$, the irreducible components of $W(a_i)$ are
  exactly the Zariski connected components of $W(a_i)$ and hence
  mutually disjoint.
\end{corollary}
\begin{proof}
  Corollary~\ref{c:2} follows immediately from Theorem~\ref{t:2}
  taking into account~\cite[Chapter~1, \S~9, Remark]{Matsumura86}.
\end{proof}
Let $a\in \C^{(s-p)\times s}$ be generic and $1\le i \le r$. Following
the Thom-Porteous formula we may express the rational equivalence
class of $W(a_i)$ in terms of the Chern classes of $E$ (see
\cite[Theorem~14.4]{Fulton84}, and, in case that $W(a_i)$ is a polar
variety, the proof of~\cite[Proposition~1.2]{Piene78}). This
argumentation yields the following statement.
\begin{theorem}\label{t:3}
  Let $a,b\in \C^{(s-p)\times s}$ be generic matrices and let $1\le i
  \le r$. Then the subvarieties $W(a_i)$ and $W(b_i)$ of $W$ are
  rationally equivalent.
\end{theorem}
In the case of generic polar varieties Theorem~\ref{t:3} corresponds
to~\cite[Proposition~1.2]{Piene78}. It is not too hard to prove by
elementary techniques the \emph{algebraic} equivalence of $W(a_i)$ and
$W(b_i)$. However, the proof of their \emph{rational} equivalence
seems to be out of the reach of direct arguments.
\subsection{Geometric tools}
The following two technical statements
will be used in Section~\ref{s:2}, where we describe our main
algorithm.
\par
For $1\le k_1 < \cdots < k_p\le s$ we denote by $\Delta_{k_1\klk k_p}$
the $p$-minor of $F$ given by the columns numbered $k_1\klk k_p$, and
for the columns numbered $1\le l_1 <\cdots < l_{s-i}\le s$ that
contain $k_1\klk k_p$, we denote by $m_{l_1\klk l_{s-i}}$ the
$(s-i)$-minor of $T(a_{i+1})$ given by the columns numbered $l_1\klk
l_{s-i}$ (in case $s=r+p$ and $i=r$ we have $m_{l_1\klk
  l_p}=\Delta_{k_1\klk k_p}$). The following lemma is borrowed
from~\cite[Section~4.3]{BaGiHeLePa12}.
\begin{lemma}\label{l:2}
  Let $1\le i \le r$ and let $C$ be an irreducible component of
  $W(a_i)_{\Delta_{k_1\klk k_p}}$.  Then the polynomial $m_{l_1\klk
    l_{s-i}}$ does not vanish identically on $C$.
\end{lemma}
\begin{proof}
  Fix $1\le i \le r$. Without loss of generality we may assume $k_1:=1
  \klk k_p:=p$ and $l_1:=1\klk l_{s-i}:=s-i$ and hence,
  $\Delta_{k_1\klk k_p}:=\Delta$ and $m_{l_1\klk l_{s-i}}:=m_i$. By
  induction on $1\le i \le r$ one deduces from the genericity of the
  complex matrix $a$ that $m_i$ does not vanish identically on any
  irreducible component of $W_{\Delta}$. Therefore the affine variety
  $\mathcal{Y}:=W_{\Delta}\cap \{m_i=0\}$ is empty or of pure
  codimension one in $W_{\Delta}$.
  \par
  Let $1\le l_1^*<\cdots <l_{s-i}^*\le s$ be arbitrary and denote the
  $(s-i)$-minor $m_{l_1^*\klk l_{s-i}^*}$ of $T(a_{i+1})$ by $m_i^*$.
  Further, let $M_{s-i+1}^*\klk M_s^*$ be the $(s-i+1)$-minors of
  $T(a_i)$ given by the columns numbered $l_1^*\klk l_{s-i}^*$ to
  which we add, one by one, the columns numbered by the elements of
  the index set $\{1\klk s\}\setminus \{l_1^*\klk l_{s-i}^*\}$. Again,
  the genericity of $a$ implies that the intersection
  $\mathcal{Y}_{m_i^*}\cap \{M_{s-i+1}^*=0\klk M_s^*=0\}$ is empty or
  of pure codimension $i$ in $\mathcal{Y}_{m_i^*}$ and hence of pure
  codimension $i+1$ in $W_{\Delta\cdot m_i^*}$.
  \par
  Let $C$ be an irreducible component of $W(a_i)_{\Delta}$. From
  Proposition~\ref{p:1} we deduce that $C$ is not contained in
  $W(a_{i+1})_{\Delta}$. This implies that there exists an
  $(s-i)$-minor $m_i^*$ of $T(a_{i+1})$ with $C_{m_i^*}\neq \emptyset$. The
  corresponding $(s-i+1)$-minors $M_{s-i+1}^*\klk M_s^*$ of $T(a_i)$
  define in $W_{\Delta\cdot m_i^*}$ a variety which contains
  $C_{m_i^*}$ as irreducible component. Hence, $C_{m_i^*}$ is a subset
  of $\{M_{s-i+1}^*=0\klk M_s^*=0\}$. Suppose now that $m_i$ vanishes
  identically on $C$. Then $\mathcal{Y}_{m_i^*}$ contains $C_{m_i^*}$
  and is in particular nonempty. Since $C_{m_i^*}$ is contained in
  $\mathcal{Y}_{m_i^*}\cap \{M_{s-i+1}^*=0\klk M_s^*=0\}$ we conclude
  that the codimension of $C_{m_i^*}$ in $W_{\Delta\cdot m_i^*}$ is at
  least $i+1$.
  \par
  On the other hand, Proposition~\ref{p:1} implies that the
  codimension of $C_{m_i^*}$ in $W_{\Delta\cdot m_i^*}$ is~$i$. This
  contradiction proves that $m_i$ cannot vanish identically on $C$.
\end{proof}
Suppose that the quasi-affine variety $V$ is embedded in the affine
space $\A ^n$ and that the Zariski closure of $V$ in $\A ^n$ can be
defined by the polynomials of $\C[\xon]$ of degree at most
$d$. Furthermore suppose that for each $1\le i\le p$ and $1\le j\le s$
there is given a polynomial $F_{i,j}\in \C[\xon]$ of degree at most $d$
such that the entry $f_{i,j}$ of the matrix $F$ is the restriction of
$F_{i,j}$ to $V$.
\par
Let $b_1,\dots, b_{r+1}\in\C^{s\times s}$ be regular matrices. We call
$(b_1,\dots, b_{r+1})$ a \emph{hitting sequence} for $V$ and $F$ if
the following property holds: there exist $p$-minors
$\bfs{\Delta}_1\klk \bfs{\Delta}_{r+1}$ of the matrices $F\cdot
b_1,\dots,F\cdot b_{r+1}\in\C[V]^{p\times s}$ respectively, such that
for any point $x$ of $W$ at least one of the minors $\bfs{\Delta}_t$
($1\le t\le r+1$) does not vanish at $x$. The following lemma is
reminiscent of~\cite[Theorem~4.4]{HeSc82}.
\begin{lemma} \label{l:3}
  Let $\kappa:=(4pd)^{2n}$ and let
  $\mathcal{K}:=\{1,\ldots,\kappa\}$.  Then the set
  $(\mathcal{K}^{s\times s})^{r+1}$ contains at least
  $\kappa^{s^2(r+1)}(1-4^{-n})$ hitting sequences for $V$ and $F$.
\end{lemma}
\begin{proof}
  For $1\le t\le r+1$ and $1\le k,l\le s$ let $B_{k,l}^t$ be new
  indeterminates over $\C$ and let $\bfs{B}_t:=(B^t_{k,l})_{1\le
    k,l\le s}$. Furthermore, let $\bfs{\Delta}_t\in\C[V][\bfs{B}_t]$
  be the $p$-minor of $F\cdot \bfs{B}_t$ given by the first $p$
  columns of $F\cdot \bfs{B}_t$.
  \par
  Consider an arbitrary point $x$ of $W$. Without loss of generality
  we may suppose $\Delta(x)\not=0$. Fix for the moment $1\le t\le r+1$
  and consider the matrix $\bfs{C}_t$ obtained from $\bfs{B}_t$ by
  substituting zero for $B_{k,l}^t$ for any $(k,l)$ with $p+1\le k\le
  s$ and $1\le l\le p$, namely
  \[
  \bfs{C}_t:=\left[
    \begin{array}{cccccc}
      B_{1,1}^t & \cdots & B_{1,p}^t & B_{1,p+1}^t & \dots & B_{1,s}^t \\
      \vdots &  & \vdots & \vdots &  & \vdots \\
      B_{p,1}^t & \cdots & B_{p,p}^t & B_{p,p+1}^t & \dots & B_{p,s}^t \\
      0 & \cdots & 0 & B_{p+1,p+1}^t & \dots & B_{p+1,s}^t  \\
      \vdots &  & \vdots & \vdots &  & \vdots \\
      0 & \cdots & 0 & B_{s,p+1}^t & \dots & B_{s,s}^t   \\
    \end{array}
  \right].
  \]
  It is easy to see that the left $p$-minor
  $\bfs{\Delta}_t(x,\bfs{C}_t)$ of the matrix $F\cdot \bfs{C}_t$ is of
  the form $\Delta(x)$ times a nonzero polynomial of $\C[\bfs{B}_t]$.
  In particular, $\bfs{\Delta}_t(x,\bfs{C}_t)$ is a polynomial of
  positive degree. We conclude now \emph{a fortiori} that for any
  $x\in W$ the polynomial $\bfs{\Delta}_t(x,\bfs{B}_t)$ is of positive
  degree.
  \par
  We consider now the incidence variety $\mathcal{H}\subset
  W\times(\A^{s\times s})^{r+1}$ defined by the vanishing of
  $\bfs{\Delta}_1,\dots,\bfs{\Delta}_{r+1}$. Let $\pi$ be the
  projection of $\mathcal{H}$ into $(\A^{s\times s})^{r+1}$. It is not
  difficult to see that $\mathcal{H}$ is equidimensional of dimension
  $s^2(r+1)-1$. In order to show this, we proceed recursively. Let
  $W_0$ be an arbitrary irreducible component of $W$. Since the
  polynomial $\bfs{\Delta}_1(x,\bfs{B}_1)$ has positive degree in the
  variables $\bfs{B}_1$ for any point $x\in W$, the variety
  $\big(W_0\times (\A^{s\times s})^{r+1}\big)\cap\{\bfs{\Delta}_1=0\}$
  must be equidimensional of dimension $r+s^2(r+1)-1$. Moreover, each
  irreducible component of this variety has the form $W_1\times
  (\A^{s\times s})^r$, where $W_1$ is an irreducible component of
  $\big(W_0\times\A^{s\times
    s}\big)\cap\{\bfs{\Delta}_1=0\}$. Applying this argument
  recursively for each polynomial $\bfs{\Delta}_t$ we conclude that
  $\bfs{\Delta}_1,\dots,\bfs{\Delta}_{r+1}$ constitute a secant family
  for the variety $W\times (\A^{s\times s})^{r+1}$ (recall that
  $\bfs{\Delta}_1,\dots, \bfs{\Delta}_{r+1}$ are polynomials in
  disjoint groups of indeterminates). Hence the incidence variety
  $\mathcal{H}$ is equidimensional of dimension
  $r+s^2(r+1)-(r+1)=s^2(r+1)-1$.
  \par
  In particular, we infer that the Zariski closure of
  $\pi(\mathcal{H})$ in $(\A ^{s\times s})^{r+1}$ has dimension at
  most $s^2(r+1)-1$ and therefore it is a proper closed subvariety of
  $(\A^{s\times s})^{r+1}$. Observe that the zero-dimensional variety
  $\pi(\mathcal{H})\cap(\mathcal{K}^{s\times s})^{r+1}$ contains all
  sequences of $(\mathcal{K}^{s\times s})^{r+1}$ which are not hitting
  for $V$ and $F$.
  \begin{claim}
    $\#\big(\pi(\mathcal{H})\cap(\mathcal{K}^{s\times
      s})^{r+1}\big)\le (2pd)^{2n}\kappa^{s^2(r+1)-1}$.
  \end{claim}
  \begin{proof}[Proof of the Claim]
    Observe $\pi^{-1}\big(\pi(\mathcal{H})\cap(\mathcal{K}^{s\times
      s})^{r+1}\big)=\mathcal{H}\cap\big(\A^n\times
    (\mathcal{K}^{s\times s})^{r+1}\big)$. Let $C_1\klk C_m$ be the
    irreducible components of $\mathcal{H}\cap\big(\A^n\times
    (\mathcal{K}^{s\times s})^{r+1}\big)$. As the image under $\pi$ of
    each component $C_j$ of $\mathcal{H}\cap\big(\A^n\times
    (\mathcal{K}^{s\times s})^{r+1}\big)$ is a point of
    $\pi(\mathcal{H})\cap(\mathcal{K}^{s\times s})^{r+1}$ we conclude
    \begin{equation}\label{eq:aux_hitting_seq_1}
      \#\big(\pi(\mathcal{H})\cap(\mathcal{K}^{s\times
        s})^{r+1}\big)\le m\le \sum_{i=1}^m\deg \overline{C_i}= \deg\big(
      \overline{\mathcal{H}}\cap\big(\A^n\times (\mathcal{K}^{s\times
        s})^{r+1}\big)\big)
    \end{equation}
    (here $\overline{C_i}$ denotes the Zariski closure of $C_i$ in $\A
    ^n\times (\A^{s\times s})^{r+1}$). It is easy to see that the
    affine variety $\A^n\times (\mathcal{K}^{s\times s})^{r+1}$ can be
    defined by the vanishing of $s^2(r+1)$ univariate polynomials of
    degree $\kappa$. Therefore, by \cite[Proposition~2.3]{HeSc82} it
    follows that
    \begin{equation}\label{eq:aux_hitting_seq_2}\deg\big(
      \overline{\mathcal{H}}\cap\big(\A^n\times (\mathcal{K}^{s\times
        s})^{r+1}\big)\big)\le \deg\overline{\mathcal{H}} \cdot
      \kappa^{s^2(r+1)-1}
    \end{equation}
    holds. On the other hand, the B\'ezout inequality implies
    \begin{equation}\label{eq:aux_hitting_seq_3}
      \deg\overline{\mathcal{H}}\le\deg
      \overline{V}\cdot(p(d+1))^{r+1}\le (2pd)^{2n}.
    \end{equation}
    Combining (\ref{eq:aux_hitting_seq_1}),
    (\ref{eq:aux_hitting_seq_2}) and (\ref{eq:aux_hitting_seq_3}) we
    easily deduce the statement of the claim.
  \end{proof}
  Following the previous claim the probability to find a nonhitting
  sequence for $V$ and $F$ in $(\mathcal{K}^{s\times s})^{r+1}$ is at
  most
  \[
  \frac{(2pd)^{2n}\kappa^{s^2(r+1)-1}}{\kappa^{s^2(r+1)}}=
  \frac{(2pd)^{2n}}{\kappa}=\frac{(2pd)^{2n}}{(4pd)^{2n}}=
  \frac{1}{4^n}.
  \]
  This implies Lemma~\ref{l:3}.
\end{proof}
\subsection{Algebraic characterization of degeneracy loci}
Let $U_2\klk U_s$ be new indeterminates. For $1 \le i \le r$ let
$U^{(i)}$ be the $(s\times (s-i+1))$-matrix
\[
U^{(i)}:=
\left[ \begin{array}{cccc}
  1 & 0 & \cdots  & 0 \\
  U_2 & 1 & \ddots & \vdots \\
  U_3 & U_2 & \ddots & 0 \\
  \vdots & U_3 & \ddots & 1 \\
  U_i & \vdots & \ddots & U_2 \\
  \vdots & U_i &  & U_3 \\
  U_{s-1} & \vdots & \ddots & \vdots \\
  U_s & U_{s - 1} & \cdots & U_i
\end{array} \right],
\]
and let $U:= U^{(s - p +1)}$. With these notations the following
assertion holds.
\begin{lemma}\label{l:0}
  Let $1\le i \le r$. Any point $x\in V$ belongs to $W(a_i)$ if and
  only if the conditions
  \[
  \det (F(x)\cdot U)\neq 0\quad\text{and}\quad\det (T(a_i)(x)\cdot
  U^{(i)})= 0
  \]
  are satisfied identically.
\end{lemma}
\begin{proof}
  Let $x$ be any point of $V$ which satisfies the condition $\det
  (F(x)\cdot U)\neq 0$. Then $F(x)$ must be of maximal rank $p$ and
  hence $x$ belongs to $W$.
  \par
  Suppose now that $x$ belongs to $W$. Let $K$ run over all subsets of
  $\{1\klk s\}$ of cardinality $p$. Denote by $F(x)_K$ and $U_K$ the
  $p$-minors of $F(x)$ and $U$ corresponding to the columns of $F(x)$
  and rows $U$ indexed by the elements of $K$.  The Binet-Cauchy
  formula yields
  \[
  \det (F(x)\cdot U) = \sum_{\substack{K\subseteq \{1\klk s\}\\\#
      K=p}}F(x)_K U_K.
  \]
  From the proof of~\cite[Theorem~2]{KaltSau91} we deduce that for
  $K\subseteq \{1\klk s\}$, $\# K=p$ all the minors $U_K $ are
  linearly independent over $\C$. Since $x$ belongs to $W$ there
  exists a subset $K$ of $\{1\klk s\}$ of cardinality $p$ with
  $F(x)_K\neq 0$. This implies $\det (F(x)\cdot U)\neq 0$.
  \par
  Using the same kind of arguments one shows that, for $x\in V$, the
  condition $\det (T(a_i)(x)\cdot U^{(i)})=0$ is equivalent to $\rk
  T(a_i)(x)< s-i +1$. Lemma~\ref{l:0} follows now easily.
\end{proof}
We define the \emph{point finding problem associated with the pair
  $(V,F)$} as the problem to decide whether $W(a_r)$ is empty, and if
not to find all the points of the zero-dimensional degeneracy locus
$W(a_r)$.
\par
The degree of this problem is the maximal degree of the Zariski
closures of all degeneracy loci $W(a_i)$, for $1\le i \le r$, in the
ambient space $\A^n$ of $V$. Observe that this degree does not
depend of the particular generic choice of the $((s-p)\times
s)$-matrix $a$ (compare to~\cite[Section~4]{BaGiHeSaSc10}).
\section{Examples}
\label{s:ex}
\subsection{Polar varieties}
\label{s:ex1}
Let $\xon$ be indeterminates over $\C$, $1\le p \le n$, and let
$G_1\klk G_p$ be a reduced regular sequence of polynomials in
$\C[\xon]$. We denote the Jacobian of $G_1\klk G_p$ by
\[
J(G_1\klk G_p):=
\begin{bmatrix}
  \frac{\partial G_1}{\partial X_1} & \cdots & \frac{\partial
    G_1}{\partial X_n}\\
  \vdots & &\vdots\\
  \frac{\partial G_p}{\partial X_1} &\cdots & \frac{\partial
    G_p}{\partial X_n}\\
\end{bmatrix}.
\]
Fix a $p$-minor $\Delta$ of $J(G_1\klk G_p)$ and let
\[
V:=\{G_1=0\klk G_p=0\}_{\Delta}.
\]
Then $V$ is a smooth equidimensional quasi-affine subvariety of $\A^n$
of dimension $r:= n-p$.  Let $s:=r+p=n$ and let $F\in \C[V]^{p\times
  s}$ be the $(p\times s)$-matrix induced by $J(G_1\klk G_p)$ on
$V$.
\par
For a given generic complex $((s-p)\times s)$-matrix $a$ and for $1\le
i\le r$ the degeneracy locus $W(a_i)$ is the $i$-th generic (classic)
polar variety of $V$ associated with the complex $((s-p-i+1)\times
s)$-matrix $a_i$ (see details in~\cite{BaGiHeSaSc10}).
\par
Proposition~\ref{p:1}, Corollary~\ref{c:1} and Theorem~\ref{t:2} above
say that the $i$-th generic (classic) polar variety of $V$ is empty or
a normal Cohen-Macaulay subvariety of $V$ of pure codimension $i$
(compare to~\cite[Theorem~2]{BaGiHeSaSc10}). From
\cite[Section~3.1]{BaGiHeSaSc10} we deduce that such a generic polar
variety is not necessarily smooth. Hence, smoothness of our degeneracy
loci cannot be expected in general. If the coefficients of
$G_1,\dots,G_p$ and the entries of the $((n-p)\times n)$-matrix $a$
are real, and if the real trace of $\{G_1=0,\dots,G_p=0\}$ is smooth
and compact, then there exists a $p$-minor $\Delta$ of
$J(G_1,\dots,G_p)$ such that the polar varieties associated with $a$
contain real points and are therefore nonempty. The generic polar
varieties form then a strictly descending chain (see~\cite{BaGiHePa04}
and \cite[Proposition~1]{BaGiHePa05}).
\subsection{Composition of polynomial maps}
\label{s:ex2}
Let $1\le p\le n$ and $Q_1\klk Q_n$, $P_1\klk P_p$ be polynomials of
$\C[\xon]$ such that $P_1\klk P_p$ form a reduced regular
sequence. Moreover, let
\[
(G_1\klk G_p):=(P_1\klk P_p)\circ (Q_1\klk Q_n)
\]
be the composition map defined for $1\le k \le p$ by
\[
G_k(\xon):= P_k(Q_1(\xon), \ldots Q_n(\xon)).
\]
Suppose that $G_1\klk G_p$ constitute a reduced regular sequence
in $\C[\xon]$. Fix a $p$-minor $\Delta$ of the Jacobian
$J(G_1\klk G_p)$. Then
\[
V:=\{G_1=0\klk G_p=0\}_{\Delta}
\]
is a smooth quasi-affine subvariety of $\A^n$ of dimension
$r:=n-p$. The morphism defined by $(Q_1\klk Q_n)$ maps $V$ into
\[
\V:=\{P_1=0\klk P_p=0\}.
\]
We suppose that this morphism of affine varieties is dominant,
\emph{i.e.},
\[\overline{(Q_1\klk Q_n)(V)}=\V.\]
Observe that for any point $x\in V$ the variety $\V$ is smooth at
$y=(Q_1(x)\klk Q_n(x))$.
\par
Let $s:=r+p=n$ and let $F$ be the $(p\times s)$-matrix induced by
$J(P_1\klk P_p)\circ (Q_1\klk Q_n)$ on $V$. Let $a\in \C^{(s-p)\times
  s}$ be a generic complex matrix and denote by $\widetilde{W}(a_i)$
the $i$-th polar variety of $\V$ associated with $a_i$, for $1\le i
\le r$. Then we have $W=V$, and the $i$-th degeneracy locus $W(a_i)$
of $W$, namely
\begin{multline*}
  W(a_i)=\left\{x\in V \mid \rk \! \begin{bmatrix}F(x)\\
    a_i \end{bmatrix}<s-i+1\right\}=\\
  =\left\{x\in V \mid \rk \! \begin{bmatrix}
    J(P_1\klk P_p)\circ (Q_1\klk Q_n)(x)\\
    a_i
  \end{bmatrix}<s-i+1\right\},
\end{multline*}
is the $(Q_1\klk Q_n)$-preimage of $\widetilde{W}(a_i)$.
\subsection{Dominant endomorphisms of affine spaces}
\label{s:ex3}
Let $\Fon\in \C[\xon]$, $V:=\A^n$, $p:=1$, $s:=p+r=1+n$,
$F:=\begin{bmatrix}F_1\klk F_n,1\end{bmatrix}\in \C^{1\times s}$ and
let $a\in\C^{n\times s}$ be a generic complex matrix. Observe
$W=V=\A^n$ and that, for any $1\le i \le n$, the degeneracy locus
$W(a_i)$ is a closed affine subvariety of $\A^n$. We are now going to
analyze the $n$-th degeneracy locus $W(a_n)$.
\begin{lemma}\label{l:polar_var_for_endomorphism}
  The degeneracy locus $W(a_n)$ is non-empty if, and only if, the
  endomorphism $\Psi:\A^n \longrightarrow \A^n$ defined by
  $\Psi(x):=(F_1(x)\klk F_n(x))$ is dominant.  In this case the
  cardinality $\# W(a_n)$ of $W(a_n)$ equals the cardinality of a
  generic fiber of $\Psi$.
\end{lemma}
\begin{proof}
  Suppose that $W(a_n)$ is nonempty and let $x$ be a point of
  $W(a_n)$. Then there exists a $\lambda \in \C$ such that
  $(F_1(x)\klk F_n(x), 1)=\lambda(a_{1,1}\klk a_{1,n},
  a_{1,n+1})$. This implies $(F_1(x)\klk
  F_n(x))=\frac{1}{a_{1,n+1}}(a_{1,1}\klk a_{1,n})$. The right-hand
  side of this equation is therefore a generic point of $\A^n$ with a
  zero-dimensional $(F_1\klk F_n)$-fiber.  Hence, the endomorphism
  $\Psi$ of $\A^n$ is dominant.
  \par
  Suppose now that $\Psi$ is dominant. Then we may assume without loss
  of generality that there exists a point $x\in \A^n$ with
  $(F_1(x)\klk F_n(x))=\frac{1}{a_{1,n+1}}(a_{1,1}\klk a_{1,n})$. This
  implies the equation $(F_1(x)\klk F_n(x), 1)=\lambda (a_{1,1}\klk
  a_{1,n}, a_{1,n+1})$ with $\lambda=\frac{1}{a_{n+1}}$. Hence $x$
  belongs to $W(a_n)$ and thus $W(a_n)$ is not empty. Moreover, $\#
  W(a_n)$ equals the cardinality of the $(F_1\klk F_n)$-fiber of
  $\frac{1}{a_{1,n+1}}(a_{1,1}\klk a_{1,n})$.
\end{proof}
Suppose now that the morphism $\Psi$ is dominant. Then the degeneracy
loci of $(\A^n,F)$ form a descending chain
\[
\A^n\supsetneq W(a_1)\supsetneq \cdots \supsetneq W(a_n)\neq
\emptyset,
\]
where, for $1\le i < n$, the $(i+1)$-th degeneracy locus $W(a_{i+1})$
is a closed affine subvariety of $W(a_i)$ of pure codimension one in
$W(a_i)$.
\subsection{Homotopy}
\label{s:ex4}
Let $F_1\klk F_n$ and $G_1\klk G_n$ be reduced regular sequences of
$\C[\xon]$. We consider the algebraic family
\[
\{\lambda F_1+\mu G_1=0\klk \lambda F_n+\mu G_n=0\},\quad (\lambda ,
\mu) \in \C^2\setminus \{(0,0)\}
\]
as a homotopy between the zero-dimensional varieties $\{F_1=0\klk
F_n=0\}$ and $\{G_1=0\klk G_n=0\}$.
\par
We are going to analyze this homotopy. For this purpose let $V:=\A^n$,
$p:=2$, $s:=n+p=n+2$ and
\[
F:=\begin{bmatrix}
  F_1&\cdots & F_n & 1& 0\\
  G_1&\cdots & G_n & 0&1
\end{bmatrix}.
\]
Furthermore, let $a\in \C^{(s-p)\times s}$ be generically chosen.
Then we have $W=V=\A^n$ and for any $1\le i \le n$ the degeneracy
locus $W(a_i)$ is a closed affine subvariety of $\A^n$. From the
Exchange Lemma of~\cite{BaGiHeMb01} we deduce
\[
\# W(a_n)=\# \{a_{1,n+1} F_1+a_{1,n+2} G_1=a_{1,1}\klk a_{1,n+1}
F_n+a_{1,n+2}G_n=a_{1,n}\}.
\]
Thus $W(a_n)$ may be interpreted as a deformation of
\[\{a_{1,n+1} F_1+a_{1,n+2} G_1=0\klk a_{1,n+1}
F_n+a_{1,n+2}G_n=0\}.\]
\section{Algorithms}
\label{s:2}
We are going to present two procedures, namely our main algorithm
that computes an algebraic description of the set $W(a_r)$, and a
procedure to check membership to a degeneracy locus.
\subsection{Notations}
Let $n, d, p, r, q, s, L$ be integers with $r=n-q$ and $s\ge p+r$
and let $G_1\klk G_q, H, F_{k,l}$, for $1\le k \le p$ and $1\le l
\le s$, be polynomials of $\Q[\xon]$ given as outputs of an
essentially division-free arithmetic circuit $\beta$ of size $L$.
This means that $\beta$ contains divisions only by elements of $\Q$
(for details about arithmetic circuits we refer to~\cite{BuClSh97}).
\par
Let $d$ be an upper bound for the degrees of $G_1\klk G_q$ and
$F_{k,l}$, for $1\le k \le p$ and $1\le l \le s$. We suppose that
$G_1\klk G_q$ and $H$ satisfy the following two conditions:
\begin{itemize}
\item $G_1\klk G_q$ form a reduced regular sequence outside of
  $\{H=0\}$,
\item $V:=\{G_1=0\klk G_q=0\}_H$ is a smooth quasi-affine variety.
\end{itemize}
For $1\le k \le p$ and $1\le l \le s$, let $f_{k,l}\in \C[V]$ be the
restriction of $F_{k,l}$ to $V$ and let $F:=[f_{k,l}]_{1\le k \le p, 1
  \le l \le s}$. Let $\delta^*$ be the degree of the point finding
problem associated with the pair $(V,F)$, that previously has been
introduced as the maximal degree of the Zariski closures of all
degeneracy loci $W(a_i)$, $1\le i \le r$, in the ambient space
$\A^n$ of $V$. We write
\[
\delta_G:=\max\{\deg \overline{\{G_1=0\klk G_j=0\}_H} \mid 1\le j \le
q\},
\]
and $\delta:=\max\{ \delta_G, \delta^* \}$. We call $\delta$ the
\emph{system
  degree} of $G_1\klk G_q=0$, $H\neq 0$ and $[F_{k,l}]_{1\le k \le p,
  1\le l \le s}$.
\par
Fix a generic matrix $a\in \Q^{(s-p)\times s}$. We are going to
design a uniform bounded error probabilistic procedure which takes
$\beta$ as input and decides whether $W(a_r)$ is empty and, if not,
computes a description of $W(a_r)$ in terms of a primitive element.
More precisely, for a new indeterminate $T$, the procedure outputs
the coefficients of univariate polynomials $P, Q_1\klk Q_n\in \Q[T]$
such that $P$ is separable, $\deg Q_1< \deg P\klk \deg Q_n < \deg P$
and such that
\[
W(a_r)=\{(Q_1(t)\klk Q_n(t)) \mid t\in \C : P(t)=0\}
\]
holds. Following \cite[Section~3.2]{GiLeSa01}, such a description is
called a \emph{geometric resolution} of $W(a_r)$.
\par
In the sequel we refer freely to terminology, mathematical results
and subroutines of~\cite{GiLeSa01} where the first streamlined
version of the classical Kronecker algorithm was described. In order
to simplify the exposition we shall refrain from the presentation of
details which only ensure the appropriate genericity properties for
the procedure. The following account requires some familiarity with
technical aspects of the classical Kronecker algorithm. A standalone
presentation of the algorithm from the mathematical point of view is
contained in~\cite{DurLe08}.
\subsection{The main algorithm}
As first task, we compute a description of the variety $V$. For this
purpose, we use the main tools of~\cite[Algorithm~12]{GiLeSa01} in
the following way. As input we take the representation of $G_1,
\ldots, G_q$ and $H$ by the circuit $\beta$. Although the system
$G_1=0,\dots,G_q=0$ contains $n\ge q$ variables, we may execute just
the $q$ first steps of the main loop
of~\cite[Algorithm~12]{GiLeSa01}, to obtain a \emph{lifting fiber}
for $V$~\cite[Definition~4]{GiLeSa01}. This lifting fiber consists
of:
\begin{itemize}
\item the \emph{lifting system} $G_1,\ldots,G_q$,
\item an invertible $n \times n$ square matrix $M$ with rational entries
 such that the new coordinates $Y := M^{-1} X$ are in
  \emph{Noether position} with respect to $V$,
\item a rational \emph{lifting point} $z=(z_1,\ldots,z_r)$ for $V$ and
  the lifting system $G_1,\ldots,G_q$,
\item rational coefficients $\lambda_{r+1},\dots,\lambda_n$ defining a
  \emph{primitive element} $u:=
  \lambda_{r+1}Y_{r+1}+\cdots+\lambda_{n} Y_{n}$ of $V^{(z)} := V \cap
  \{ Y_1 - z_1 = 0, \ldots, Y_r - z_r = 0 \}$,
\item a polynomial $Q\in \Q[T]$ of minimal degree such that $Q(u)$
  vanishes on $V^{(z)}$,
\item $n-r$ polynomials $v_{r+1},\ldots,v_n$ of $\Q[T]$, of degree
  strictly smaller than $\deg Q$ such that the equations
  $Y_1-z_1=0,\dots,Y_r-z_r=0,Y_{r+1}-v_{r+1}(T)=0,\dots,Y_n-v_n(T)=0,
  Q(T)=0$ define a \emph{parameterization} of
  $V^{(z)}$ by the zeros of $Q$.
\end{itemize}
The computation of these items depends on the choice of at most
$O(n^2)$ parameters in $\Q$. If the parameters are chosen correctly,
the algorithm returns these items. Otherwise the algorithm fails.  The
incorrect choices of these parameters are contained in a hypersurface
whose degree is \emph{a priori} bounded (see~\cite{GiLeSa01}).
Therefore the whole procedure yields a bounded error probabilistic
algorithm (compare~\cite{Zippel79, Schwartz80}).
The error can be bounded uniformly with respect to the input parameters, whatever they are (dimension of the ambient space $n$, degree and coefficients of the input equations, etc.).\\
 We summarize the outcome in the
following statement.
\begin{lemma}\label{lm:costV}
  Let notations and assumptions be as above. There exists a uniform
  bounded error probabilistic algorithm over $\Q$ which computes a
  lifting fiber of $V$ in time $L (nd)^{O(1)}\delta_G^2$.
\end{lemma}
\begin{proof}
  Apply~\cite[Theorem~1]{GiLeSa01} taking care to perform products of
  univariate polynomials in quasi-linear time. The B\'ezout inequality
  implies $\delta_G=O(d^n)$. The complexity bound of Lemma \ref{lm:costV}
  follows now from $\delta_G^2 \log^{O(1)} (\delta_G)
  = (nd)^{O(1)} \delta_G^2$.
\end{proof}
%
%
By Lemma~\ref{l:3} we may choose with high probability of success a
hitting sequence $(b_1,\dots, b_{r+1})$ of regular integer $(s\times
s)$-matrices and $p$-minors $\bfs{\Delta}_1\klk \bfs{\Delta}_{r+1}$
of the matrices $F\cdot b_1,\dots,F\cdot b_{r+1}$ such that $W =
V_{\bfs{\Delta}_1} \cup \cdots \cup V_{\bfs{\Delta}_{r+1}}$ holds.
\begin{lemma}\label{lm:costW}
  Let notations and assumptions be as above and let be given a lifting
  fiber of $V$. There exists a uniform bounded error probabilistic
  algorithm over $\Q$ which computes lifting fibers for
  $V_{\bfs{\Delta}_1}, \ldots,V_{\bfs{\Delta}_{r+1}}$ in time $L
  (pnd)^{O(1)}\delta_G^2$.
\end{lemma}
\begin{proof}
  Let us fix $1\le j \le r+1$. The given lifting point of $V$ may
  be changed by means of~\cite[Algorithm~5]{GiLeSa01} in time $L
  (nd)^{O(1)}\delta_G^2$. We call~\cite[Algorithm~10]{GiLeSa01}
  with input the lifting fiber of $V$ and the polynomial representing
  $\bfs{\Delta}_j$ (observe that this polynomial can be evaluated
  using $L+O(p^4)$ arithmetic operations). This yields with high
  probability of success a lifting
  fiber of $V_{\bfs{\Delta}_j}$ in time $L (pnd)^{O(1)}\delta_G$
  (compare~\cite[Lemmas~14 and 15]{GiLeSa01}).
\end{proof}
If the varieties $V_{\bfs{\Delta}_1},\dots,V_{\bfs{\Delta}_{r+1}}$
are empty, then $W(a_r)$ is empty and the algorithm stops. We
suppose that this is not the case. We are now going to describe how
we decide whether $W(a_r)_{\bfs{\Delta}_j}$ is empty, and, if not,
how we compute a lifting fiber of $W(a_r)_{\bfs{\Delta}_j}$. In
order to simplify notations, we make without loss of generality the
following assumptions. Let $j:=1$, $b_1$ be the identity matrix,
$\bfs{\Delta}_1:= \Delta$ and $V_{\Delta}= W_{\Delta} \neq
\emptyset$.
\begin{lemma}\label{lm:costW1}
  Let notations and assumptions be as before. For a given lifting
  fiber of $V_{\Delta}=W_{\Delta}$ there exists a uniform bounded
  error probabilistic algorithm over $\Q$ which computes a lifting
  fiber of $W(a_1)_{\Delta \cdot m_1}$
  in time $L (s n d)^{O(1)}\delta^2$.
\end{lemma}
\begin{proof}
  Observe that $W(a_1)= V_{\Delta} \cap \{ \det (T(a_1)) = 0 \}$ holds.
  By Proposition~\ref{p:1} the polynomial representing $\det(T(a_1))$
  does not vanish identically on any irreducible component of
  $V_{\Delta}$. Thus we may use~\cite[Algorithms~2, 4, 5, 6
  and~11]{GiLeSa01} in order to compute a lifting fiber of
  $W(a_1)_{\Delta \cdot m_1}$. Since the polynomials representing
  $\Delta$ and $m_1$ have degrees bounded by $p d$ and can be evaluated
  in time $L + O(s^4)$, Lemma \ref{lm:costW1} follows
  from~\cite[Lemmas~6, 14, and~16]{GiLeSa01}.
\end{proof}
{From} Lemma~\ref{l:2} we deduce that emptiness of $W(a_1)_{\Delta
\cdot m_1}$ implies that of $W(a_1)_{\Delta}$ and hence that of
$W(a_r)_{\Delta}$.

Let $1\le i<r$ and assume that we have computed a lifting fiber of
$W(a_i)_{\Delta \cdot m_i}$.
\begin{lemma}\label{lm:costWi}
  Let notations and assumptions be as before. There exists a uniform
  bounded error probabilistic algorithm over $\Q$ which decides whether
  $W(a_{i+1})_{\Delta}$ is empty and, if not, computes a lifting fiber of
  $W(a_{i+1})_{\Delta \cdot m_{i+1}}$ in time $L (s n d)^{O(1)}
  \delta^2$.
\end{lemma}
\begin{proof}
  In Section~\ref{s:1.2} we have seen that the equations $M_{s-i+1}=0
  \klk M_s=0$ intersect transversally at any of their common zeros
  belonging to $V_{m_i}$. Therefore $G_1\klk G_q$ and the polynomials
  representing $M_{s-i+1}\klk M_s$ form a reduced
  regular sequence outside of $\{ \Delta \cdot m_i=0 \}$. From
  Lemma~\ref{l:2} we deduce that $m_i$ does not vanish identically on
  any irreducible component of $W(a_i)_{\Delta}$. Hence the given lifting
  fiber of $W(a_i)_{\Delta \cdot m_i}$ is also a lifting fiber of
  $W(a_i)_{\Delta}$, and $G_1\klk G_q, M_{s-i+1}\klk M_s$
  can be used as a lifting system of the lifting fiber of
  $\overline{W(a_i)_{\Delta \cdot m_i}}$.
  \par
  Applying successively~\cite[Algorithms~4, 5, and~6]{GiLeSa01} we
  produce a Kronecker parame\-terization of a suitable curve $C$ in
  $\overline{W(a_i)_{\Delta \cdot m_i}}$ on which $\Delta
  \cdot m_i$ does not vanish identically.
  \par
  Then we apply~\cite[Algorithm~2]{GiLeSa01} to $C$, $m_i$ and $H
  \cdot \Delta \cdot m_{i+1}$ in order to obtain a lifting fiber of
  $(C \cap \{m_i= 0\})_{\Delta \cdot m_{i+1}}$.
  \par
  Let $N_{s-i},\ldots,N_s$ be the polynomials representing the
  $(s-i)$-minors of $T(a_{i+1})$, given by the columns numbered
  $1,\ldots, s-i-1$ to which we add, one by one, the columns
  $s-i,\ldots, s$. In a way very similar
  to~\cite[Algorithm~10]{GiLeSa01} we can remove the points of the
  given lifting fiber $(C \cap \{m_i= 0\})_{\Delta \cdot m_{i+1}}$
  which are not zeros of $N_{s-i},\ldots,N_s$ in order to obtain a
  lifting fiber of $\overline{W(a_{i+1})_{\Delta \cdot m_{i+1}}}$.
  The time cost of the whole procedure is a consequence
  of~\cite[Lemmas~3, 6, 14, and~16]{GiLeSa01}
\end{proof}
Applying Lemma~\ref{lm:costW1} and Lemma~\ref{lm:costWi} iteratively
we obtain a lifting fiber of the zero-dimensional variety
$W(a_r)_{\bfs{\Delta} \cdot m_r}$ and hence, by Lemma~\ref{l:2},
of $W(a_r)_{\Delta}$.
%
%
Combining all previously described procedures we obtain the
announced main algorithm.
\begin{theorem}\label{t:4}
  Let $n, d, p, r, q, s, L, \delta \in \N$ with $r=n-q$ and
  $s\ge p+r$ be arbitrary and let $G_1\klk G_q, H$ and $F_{k,l}$, for
  $1\le k \le p$, $1\le l \le s$, be polynomials of $\Q[\xon]$ of
  degree at most $d$. Suppose that $G_1\klk G_q$ form a reduced
  regular sequence outside of $\{H=0\}$, the variety $V:=\{G_1=0\klk
  G_q=0\}_H$ is smooth, and the system degree of $G_1=0\klk
  G_q=0$, $H\neq 0$ and $[F_{k,l}]_{1\le k \le p, 1\le l \le s}$ is at
  most $\delta$.
  \par
  Furthermore, suppose that these polynomials are given as outputs of
  an essentially division-free arithmetic circuit $\beta$ in $\Q[\xon]$ of size
  at most $L$. Let $a\in \Q^{(s-p)\times s}$ be a generic matrix. Then
  there exists a uniform bounded error probabilistic algorithm over
  $\Q$ which decides from the input $\beta$ in time
  $L(snd)^{O(1)}\delta^2=(s(nd)^n)^{O(1)}$ whether $W(a_r)$
  is empty and, if this is not the case, computes a geometric
  resolution of $W(a_r)$ (here arithmetic operations and comparisons
  in $\Q$ are taken into account at unit costs.)
\end{theorem}
\begin{proof}
  This result is essentially a consequence of Lemmas~\ref{lm:costV},
  \ref{lm:costW}, \ref{lm:costW1}, \ref{lm:costWi}. In fact, first we
  obtain for all $1\le j\le r+1$ lifting fibers of $W(a_r)_{\bfs{\Delta}_j}$.
  Then we change back the variables of
  the lifting fibers and find a primitive element common to all the
  fibers by means of~\cite[Algorithm~6]{GiLeSa01} at a total cost of
  $O((snd)^{O(1)}\delta^2)$.
  \par
  By means of classical greatest common divisor computations, we remove the
  points of $W(a_r)_{\bfs{\Delta}_2}$ that belong already to
  $W(a_r)_{\bfs{\Delta}_1}$. Then we remove the points of
  $W(a_r)_{\bfs{\Delta}_3}$ that belong already to
  $W(a_r)_{\bfs{\Delta}_1}$ and $W(a_r)_{\bfs{\Delta}_2}$. Recursively
  we remove the points of $W(a_r)_{\bfs{\Delta}_j}$ that belong
  already to $W(a_r)_{\bfs{\Delta}_k}$ for $k < j$. The total cost of
  these operations remains bounded by $O((nd)^{O(1)} \delta)$.
\end{proof}
\begin{remark}\label{r:1}
  For any $n, d, p, r, q, s, L, \delta \in \N$ with $r=n-q$ and $s\ge
  p+r$ the probabilistic algorithm of Theorem~\ref{t:4} may be
  realized by an algebraic computation tree of depth
  $L(snd)^{O(1)}\delta^2= (s(nd)^n)^{O(1)}$ that depends on parameters
  which may be chosen randomly. The proof of this statement requires a
  suitable refinement of Lemma~\ref{l:3} above in the spirit
  of~\cite[Theorem~4.4]{HeSc82}, which exceeds the scope of this
  paper.
\end{remark}
\subsection{Checking membership to a degeneracy locus}
Finally, we are going to consider the computational task to decide
for any $x\in \A^n$ and any $1\le i \le r$ whether $x$ belongs to
$W(a_i)$.
\begin{proposition}\label{p:2}
  Let notations and assumptions be as before, let $1\le i \le r$,
  and let $\Q[\alpha]$ be an algebraic extension of $\Q$ of degree
  $e$, given by the minimal polynomial of $\alpha$. Then, there exists
  a bounded error probabilistic algorithm $\mathcal{B}$ which, for any
  point $x \in \Q[\alpha]^n$, decides in sequential time $O (e
  (L+s^{O(1)}) \log^{O(1)} e)=(e s d^n)^{O(1)}$ whether
  $x$ belongs to $W(a_i)$.
  \par
  For any $n, d, p, r, q, s, L \in \N$ with $r=n-q$ and $s\ge
  p+r$ the probabilistic algorithm $\mathcal{B}$ may be realized by an
  essentially division-free arithmetic circuit of size
  $O(e (L+s^{O(1)}+n)^2 \log^{O(1)} e)=(e sd^n)^{O(1)}$
  that depends on parameters which may be chosen randomly.
\end{proposition}
\begin{proof}
  Checking membership of $x$ to $V$ takes $O(L)$ operations in
  $\Q[\alpha]$. Each field operation in $\Q[\alpha]$ can be performed
  by $e\log^{O(1)} e$ operations in $\Q$. Lemma~\ref{l:0} justifies now
  the following probabilistic test whether $x\in V$ belongs to $W(a_i)$.
  With a high probability of
  success we can choose values $u_i$ for the variables $U_i$, so that
  if we write $u^{(i)}$ (resp. $u$) for the corresponding
  specialization of $U^{(i)}$ (resp. of $U$), the test becomes the
  verification of the conditions
  \begin{equation}
    \label{eq:testu}
    \det (F(x)\cdot u)\neq 0\quad\text{and}\quad\det (T(a_i)(x)\cdot
    u^{(i)})= 0.
  \end{equation}
  This leads to an additional cost of $e(L + s^{O(1)})\log^{O(1)} e$.
  \par
  The second part of Proposition~\ref{p:2} is a direct consequence
  of~\cite[Theorem~4.4]{HeSc82}. 
\end{proof}
\subsection{Example}
We are going to exemplify how our main algorithm runs on the
following example. Let $n:=3$, $q:=1$, $G_1:= X_1^2 + X_2^2 +
X_3^2$, $H := X_1 X_2 X_3$, $p:=1$, $s:=3$, $F_{1,1} := X_1$,
$F_{1,2}:= X_1 X_2 + X_2^2$, $F_{1,3}:= X_1 X_3$ and $a :=
\begin{bmatrix}
  1 & 2 & 3\\
  2 & 1 & 3
\end{bmatrix}$. The variety $V= \{G_1=0\}_H$ is smooth of dimension
$r:=2$. The algorithm starts representing a lifting fiber for $V$ in
the following way:
\begin{itemize}[itemsep=0mm]
\item $G_1$ as lifting system,
\item $Y_1:= X_1 - X_2$, $Y_2:= X_2$, $Y_3:= X_3$ as new
  coordinates,
\item $(-1,-1)$ as lifting point,
\item $u:= Y_3$ as primitive element,
\item $Q:= T^2 + 5$ as the minimal polynomial of $u$,
\item $v_3:= T$ as parameterization.
\end{itemize}
For the sake of simplicity, all the random choices made by the
Kronecker routines are kept simple throughout this example. We took
care to verify that they are generic enough to ensure the
correctness of the computations.
\par
As hitting sequence $b_1$, $b_2$, and $b_3$, we choose the identity
matrix and take $\bfs{\Delta}_1= X_1$, $\bfs{\Delta}_2= X_1 X_2 +
X_2^2$, $\bfs{\Delta}_3= X_1 X_3$. Since $V_{\bfs{\Delta}_1}=
V_{\bfs{\Delta}_3} = V$ holds, it is sufficient to carry out the
computations for $V_\Delta$, where $\Delta:= \bfs{\Delta}_1$. Hence
our lifting fiber is also a lifting fiber for $V_\Delta$.
\par
The lifting curve $V_\Delta \cap \{Y_1 = -1\}$ is described by the
following equations: $T^2 + 2 Y_2^2 - 2 Y_2 + 1 = 0$, $Y_3= T$, $Y_1=
1$. The intersection of this curve with $\det(T(a_1))= 3 X_1 + 3 X_1
X_2 + 3 X_2^2 - 3 X_1 X_3= 3 (Y_1 + Y_2) (1 + Y_2 - Y_3) + 3 Y_2^2$
leads to the following lifting fiber for $W(a_1)$:
\begin{itemize}[itemsep=0mm]
\item $G_1$, $\det(T(a_1))$ as lifting system,
\item $Y_1:= X_1 - X_2$, $Y_2:= X_2$, $Y_3:= X_3$ as new coordinates,
\item $(-1)$ as lifting point,
\item $u:= Y_2$ as primitive element,
\item $Q:= 6 T^4 - 6 T^3 + 3 T^2 - 4 T + 2$ as the minimal
  polynomial,
\item $v_2:= T$, $v_3:= -6 T^3 - T + 3$ as parameterization.
\end{itemize}
We verify that none of the points of this fiber annihilates $H$,
$\Delta$ or $m_1:= 2 X_1 - X_1 X_2 - X_2^2$. Hence this lifting
fiber is also a lifting fiber for $W(a_1)_{\Delta \cdot m_1}$.
\par
The lifting curve for $W(a_1)_{\Delta \cdot m_1}$ is described by the
equations: $P(T):= 6 T^4 + (10 Y_1 + 4) T^3 + (8 Y_1^2 + 6 Y_1 + 1)
T^2 + (4 Y_1^3 + 2 Y_1^2 + 2 Y_1) T + Y_1^4 + Y_1^2= 0$, $Y_2= T$,
$P'(T) Y_3= (-8 Y_{1} + 4) T^3 + (-14 Y_1^2 + 10 Y_1) T^2 + (-10 Y_1^3
+ 8 Y_1^2) T - 2 Y_1^4 + 2 Y_1^3$.  The intersection of this curve
with the hypersurface $\{m_1 = 0\}$ yields the following set of
points:
\begin{eqnarray*}
\{ Y_1^5 + 4 Y_1^4 + 31 Y_1^3 + 72 Y_1^2 + 198 Y_1 = 0,\\
   Y_2= \frac{-1}{198} Y_1^4 + \frac{7}{198} Y_1^3 - \frac{1}{22}
   Y_1^2 + \frac{3}{22} Y_1,\\
   Y_3= \frac{-1}{66} Y_1^4 - \frac{2}{33} Y_1^3 - \frac{31}{66} Y_1^2 -
\frac{12}{11} Y_1 \}.
\end{eqnarray*}
We observe that $(0,0,0)$ is the only point of this set that
annihilates $\Delta$ or $H$. Therefore the lifting fiber for
$W(a_2)_{\Delta \cdot m_2}$ we find is represented by:
\begin{itemize}[itemsep=0mm]
\item $G_1$, $\det(T(a_1))$, $m_1$ as lifting system,
\item $Y_1:= X_1 - X_2$, $Y_2:= X_2$, $Y_3:= X_3$ as new coordinates,
\item $u:= Y_1$ as primitive element,
\item $Q:= T^4 + 4 T^3 + 31 T^2 + 72 T + 198$  as the minimal
  polynomial,
\item $v_1:= T$, $v_2:= \frac{1}{18} T^3 + \frac{1}{9} T^2 + \frac{1}{2}
  T + 1$, $v_3:= 3$ as parameterization.
\end{itemize}
\par
We have implemented our main algorithm within the \texttt{C++}
library \texttt{geomsolvex} of
\textsc{Mathemagix}~\cite{Mathemagix}. In fact this implementation
uses the strategy described in~\cite[Section~7.3]{GiLeSa01}: we
first choose a suitable prime number $p$ that fits a machine word,
compute the degeneracy locus modulo $p$ and then lift the geometric
resolution in order to recover the solutions over the rational
numbers.
\section{Applications}
\label{s:3}
In this section we complete the examples of Subsections~\ref{s:ex1}
and~\ref{s:ex4}. The two other examples of Section~\ref{s:ex} may be
adapted in a straightforward way to the context of
Theorem~\ref{t:4}. We refrain from presenting details.
\subsection{Polar varieties}\label{s:polarvariety}
We consider first a somewhat modified version of the example of
Subsection~\ref{s:ex1}. Let $n\in \N$, $1\le p \le n$ and $r:=n-p$ and
let $G_1\klk G_p$ be a reduced regular sequence of polynomials of
$\Q[\xon]$. We suppose that these polynomials are given by an
essentially division-free arithmetic circuit in $\Q[\xon]$ of size $L$. From
Lemma~\ref{l:3} we deduce that we may choose a hitting sequence
$(b_1\klk b_{r+1})$ of regular matrices of $\Z^{n\times n}$ for
$\{G_1=0\klk G_p=0\}$ and the restriction of the Jacobian $J(G_1\klk
G_p)$ to this variety. This yields $p$-minors $\bfs{\Delta}_1\klk
\bfs{\Delta}_{r+1}$ of $J(G_1\klk G_p)\cdot b_1\klk J(G_1\klk
G_p)\cdot b_{r+1}$ such that
\[
\bigcup_{1\le j \le r+1}\{G_1=0\klk G_p=0\}_{\bfs{\Delta}_j}
\]
is the regular locus of $\{G_1=0\klk G_p=0\}$.
\par
Let $H:=\sum_{1\le j \le r+1}\bfs{\Delta}_j^2$ and assume that
\[
\Gamma:=\{G_1=0\klk G_p=0\}\cap \R^n
\]
is nonempty, smooth and compact. Let $V:=\{G_1=0\klk G_p=0\}_H$ and
let $F$ be the restriction of $J(G_1\klk G_p)$ to $V$. Then $V$ is
nonempty, equidimensional of dimension $r$, smooth and contains
$\Gamma$. From~\cite[Proposition~1]{BaGiHePa04} or
\cite[Proposition~1]{BaGiHePa05}, we conclude that, for $a\in
\Q^{r\times n}$ generic, $W(a_r)$ contains for each connected
component of $\Gamma$ a real point.  Let $\delta$ be the system
degree of $G_1=0\klk G_p=0$, $H\neq 0$ and $J(G_1\klk G_p)$. Then
Theorem~\ref{t:4} implies that we can compute a sample point for any
connected component of $\Gamma$ in time $L(nd)^{O(1)}\delta^2 =
(nd)^{O(n)}$. This result improves the complexity bound
of~\cite[Theorem~11]{BaGiHePa04} and~\cite[Theorem~13]{BaGiHePa05}
by a factor of $\binom{n}{p}$.
\subsection{Dominant endomorphisms of affine spaces}
We treat now the example of Subsection~\ref{s:ex3} in the spirit of
Theorem~\ref{t:4}. Let $F_1\klk F_n$ be in $\Q[\xon]$ such that
$\iFon$ defines a bi-rational endomorphism of $\A^n$. Suppose that
$\Fon$ are given by an essentially division-free arithmetic circuit
in $\Q[\xon]$ of size $L$. Let $\alpha=(\alpha_1\klk \alpha_n)\in
\Q^n$ be generic. Then Theorem~\ref{t:4} can be used to compute a
geometric solution of the polynomial equation system
$F_1-\alpha_1=0\klk F_n-\alpha_n=0$ in time $L(nd)^{O(1)}\delta^2 =
(nd)^{O(n)}$, where $\delta$ is the degree of the point finding
problem associated with $(\A^n, [F_1\klk F_n,1])$. The main outcome
of this result is that we may consider this degree as a natural
invariant of the endomorphism of $\A^n$ defined by $(\Fon)$.
\subsection{Timings}
In this final subsection we report on timings obtained with our
software \texttt{geomsolvex}. For $n:= 3$, we consider the following
infinite family of examples, which are parametrized by an integer $N
\ge 1$. For any $1 \leq j \leq N$, let $S_j := (X_1- 4j)^2 + X_2^2 +
X_3^2 - 1$, $p := 1$, and $G_1 := S_1 \cdots S_N - \epsilon$, where
$\epsilon := 1/1000000$. It is clear that $\Gamma:=\{ G_1 = 0 \} \cap
\R^n$ is compact. On the other hand, the gradient of $G_1$ is given by
\begin{eqnarray*}
  \frac{\partial G_1}{\partial X_1} &=& 2 (G_1 + \epsilon)
  \left(\frac{X_1 - 4}{S_1} + \frac{X_1 - 8}{S_2} + \cdots +
    \frac{X_1 - 4N}{S_N}\right),\\
  \frac{\partial G_1}{\partial X_2} &=&  2 X_2 (G_1 + \epsilon)
  \left(\frac{1}{S_1} + \frac{1}{S_2} + \cdots + \frac{1}{S_N} \right),\\
  \frac{\partial G_1}{\partial X_3} &=&  2 X_3 (G_1 + \epsilon)
  \left(\frac{1}{S_1} + \frac{1}{S_2} + \cdots + \frac{1}{S_N} \right).\\
\end{eqnarray*}
We observe that $\frac{1}{S_1} + \frac{1}{S_2} + \cdots +
\frac{1}{S_N}$ does not vanish on $\Gamma$. In fact the terms of this
sum are necessarily positive on $\Gamma$ since the open balls defined
by $S_j < 0$, $1 \leq j \leq N$, are all disjoint and $S_1 S_2 \cdots
S_N$ is positive on $\Gamma$. Hence $\Gamma$ is smooth at any point
$(x_1,x_2,x_3) \in \Gamma$ with $x_2 \neq 0$ or $x_3 \neq 0$. Thus on
singular points of $\Gamma$ the discriminant of the univariate
polynomial $G_1(X_1,0,0)$ vanishes. Taking this in mind we verified by
a simple computation that $\Gamma$ has no singular point for the
values of $N$ considered in our timings.
\par
In order to make dependent the equation $G_1=0$ from generic
coordinates, we replaced the variables $X_1$,
$X_2$ and $X_3$ by $3 X_1 + 5 X_2 + 7 X_3$, $X_1 - X_2 + X_3$, and $-
X_1 + 2 X_2 + 5 X_3$ respectively. Finally for $a$ we took $
\left(\begin{matrix}
    1 & 17 & 7\\
    11 & 23 & 13
\end{matrix}\right).$
We used our software \texttt{geomsolvex} described in
Section~\ref{s:polarvariety} and computed at least one point per
connected component of $\Gamma$. Timings are reported in
Table~\ref{tab:timings}. We used the \textsc{SVN} revision number~8738
of \textsc{Mathemagix} and compared with version~3.21 of the
\textsc{RAGLib} library developed in \textsc{Maple}~(TM) by M.~Safey
El Din~\cite{raglib}, which in its turn relies on the \textsc{FGb}
version~1.58 of J.-C.~Faug{\`e}re~\cite{fgb}. Our platform uses one
core of an Intel(R) Xeon(R) CPU X5650 at 2.67~GHz and disposes of
48~GB. We observed that \textsc{RAGLib} is much faster in small
input sizes. Nevertheless its cost increases faster than the one of our
probabilistic algorithm.
\begin{table}[t]
  \centering
  \begin{tabular}{|l|c|c|c|c|c|c|c|c|}
    \hline
    N & 3 & 4 & 5 & 6 & 7 & 8 & 9 \\
    \hline
    \textsc{Mathemagix} & 30 & 79 & 174 & 383 & 729 & 1380 & 2250 \\
    \hline
    \textsc{RAGLib} & 1.2 & 3.1 & 19 & 126 & 748 & 3202 & 13021 \\
    \hline
  \end{tabular}
  \caption{Timings for polar varieties, in seconds}
  \label{tab:timings}
\end{table}
\par
\medskip
\noindent
\textbf{Acknowledgment.} The authors wish to thank Antonio Campillo
(Valladolid, Spain) for stimulating conversations on the subject
of this paper.
%
%
\providecommand{\bysame}{\leavevmode\hbox
to3em{\hrulefill}\thinspace}
\providecommand{\MR}{\relax\ifhmode\unskip\space\fi MR }
\providecommand{\MRhref}[2]{%
  \href{http://www.ams.org/mathscinet-getitem?mr=#1}{#2}
} \providecommand{\href}[2]{#2}

\end{document}